\documentclass{aomart-ppn}
\usepackage{amsmath,amsfonts,amssymb,graphicx,setspace}
\usepackage{hyperref}
\newcommand{\R}{{\mathbb R}}

\newcommand{\be}[1]{\begin{equation}\label{#1}}
\newcommand{\ee}{\end{equation}}
\renewcommand{\(}{\left(}
\renewcommand{\)}{\right)}
\newcommand{\ird}[1]{\int_{\R^d}{#1}\;dx}
\newcommand{\nrm}[2]{\left\|#1\right\|_{#2}}
\def\cprime{$'$}

\newcommand{\K}[1]{\mathsf K_{#1}}
\renewcommand{\L}{\mathrm L}
\newcommand{\sphere}{{\mathbb S^{d-1}}}

\newtheorem{theorem}{Theorem}
\newtheorem{lemma}[theorem]{Lemma}

\newtheorem{corollary}[theorem]{Corollary}
\newtheorem{proposition}[theorem]{Proposition}
\newtheorem{remark}{\sl Remark}

\usepackage{color}

\title[Improved interpolation inequalities]{Improved interpolation inequalities, relative entropy and fast diffusion equations}
\author[J. Dolbeault and G. Toscani]{Jean Dolbeault and Giuseppe Toscani}
\address[J. Dolbeault]
{Ceremade (UMR CNRS no. 7534), Universit\'e Paris-Dauphine, Place de Lattre de Tassigny, F-75775 Paris C\'edex 16, France}
\email{dolbeaul@ceremade.dauphine.fr}
\address[G. Toscani]
{University of Pavia Department of Mathematics, Via Ferrata~1, 27100 Pavia, Italy} \email{giuseppe.toscani@unipv.it}
\date{\today}

\begin{document}
\begin{abstract}
We consider a family of Gagliardo-Nirenberg-Sobolev interpolation inequalities which interpolate between Sobolev's inequality and the logarithmic Sobolev inequality, with optimal constants. The difference of the two terms in the interpolation inequalities (written with optimal constant) measures a distance to the manifold of the optimal functions. We give an explicit estimate of the remainder term and establish an improved inequality, with explicit norms and fully detailed constants. Our approach is based on nonlinear evolution equations and improved entropy - entropy production estimates along the associated flow. Optimizing a relative entropy functional with respect to a scaling parameter, or handling properly second moment estimates, turns out to be the central technical issue. This is a new method in the theory of nonlinear evolution equations, which can be interpreted as the best fit of the solution in the asymptotic regime among all asymptotic profiles.
\end{abstract}

\thispagestyle{empty} \maketitle

\section{Introduction and main results}\label{Sec:Intro}

Consider the following sub-family of the Gagliardo-Nirenberg-Sobolev inequalities
\be{Ineq:GN}
\nrm f{2\,p}\le\mathcal C_{p,d}^{\rm GN}\,\nrm{\nabla f}2^\theta\,\nrm f{p+1}^{1-\theta}
\ee
with $\theta=\theta(p):=\frac{p-1}p\,\frac d{d+2-p\,(d-2)}$, $1<p\le\frac d{d-2}$ if $d\ge3$ and $1<p<\infty$ if $d=2$. Such an inequality holds for any smooth function $f$ with sufficient decay at infinity and, by density, for any function $f\in\L^{p+1}(\R^d)$ such that $\nabla f$ is square integrable. We shall assume that $\mathcal C_{p,d}^{\rm GN}$ is the best possible constant in \eqref{Ineq:GN}. In ~\cite{MR1940370}, it has been established that equality holds in \eqref{Ineq:GN} if $f=F_p$ with
\be{Eqn:Optimal}
F_p(x)=(1+|x|^2)^{-\frac 1{p-1}}\quad\forall\;x\in\R^d
\ee
and that all extremal functions are equal to $F_p$ up to a multiplication by a constant, a translation and a scaling. See Appendix~\ref{Sec:Appendix} for an expression of~$\mathcal C_{p,d}^{\rm GN}$. If $d\ge 3$, the limit case $p=d/(d-2)$ corresponds to Sobolev's inequality and one recovers the optimal functions found by T.~Aubin and G.~Talenti in \cite{MR0448404,MR0463908}. When $p\to 1$, the inequality becomes an equality, so that we may differentiate both sides with respect to $p$ and recover the euclidean logarithmic Sobolev inequality in optimal scale invariant form (see \cite{Gross75,MR479373,MR1940370} for details).

It is rather straightforward to observe that Inequality~\eqref{Ineq:GN} can be rewritten, in a non-scale invariant form, as a \emph{non-homogeneous Gagliardo-Nirenberg-Sobolev inequality:} for any $f\in\L^{p+1}\cap\mathcal D^{1,2}(\R^d)$,
\be{Ineq:GN-NonHom0}
\ird{|\nabla f|^2}+\ird{|f|^{p+1}}\ge\K{p,d}\(\ird{|f|^{2\,p}}\)^\gamma
\ee
with
\be{gamma}
\gamma=\gamma(p,d):=\tfrac{d+2-p\,(d-2)}{d-p\,(d-4)}\;.
\ee
The optimal constant $\K{p,d}$ can easily be related with $\mathcal C_{p,d}^{\rm GN}$. Indeed, by optimizing the left hand side of~\eqref{Ineq:GN-NonHom0} written for $f_\lambda(x):=\lambda^{d/(2\,p)}\,f(\lambda\,x)$ for any $x\in\R^d$, with respect to $\lambda>0$, one recovers that~\eqref{Ineq:GN-NonHom0} and \eqref{Ineq:GN} are equivalent. The detailed relation between $\K{p,d}$ and $\mathcal C_{p,d}^{\rm GN}$ can be found in Section~\ref{Sec:Conclusion}.

\medskip Define now
\[
C_M:=\(\frac {M_*}M\)^{\frac{2\,(p-1)}{d-p\,(d-4)}}\kern -4pt,\quad M_*:= \ird{\(1+|x|^2\)^{-\frac{2\,p}{p-1}}}= \pi^\frac d2\,\frac{\Gamma\big(\frac{d-p\,(d-4)}{2\,(p-1)}\big)}{\Gamma\big(\frac{2\,p}{p-1}\big)}\;.
\]
Consider next a generic, non-negative optimal function,
\[
f_{M,y,\sigma}^{(p)}(x):=\sigma^{-\frac d{4\,p}}\(C_M+\frac1\sigma\,|x-y|^2\)^{-\frac 1{p-1}}\quad\forall\;x\in\R^d
\]
and let us define the manifold of the optimal functions as
\[
\mathfrak M_d^{(p)}:=\left\{f_{M,y,\sigma}^{(p)}\;:\;(M,y,\sigma)\in\mathcal M_d\right\}\,.
\]
We shall measure the distance to $\mathfrak M_d^{(p)}$ with the functional
\[
\mathcal R^{(p)}[f]:=\inf_{g\in\mathfrak M_d^{(p)}}\ird{\left[g^{1-p}\(|f|^{2\,p}-g^{2\,p}\)-\tfrac{2\,p}{p+1}\big(|f|^{p+1}-g^{p+1}\big)\right]}\;.
\]
To simplify our statement, we will introduce a normalization constraint and assume that $f\in\L^{2\,p}(\R^2,(1+|x|^2)\,dx)$ is such that
\be{Normalization}
\frac{\ird{|x|^2\,|f|^{2\,p}}}{\(\ird{|f|^{2\,p}}\)^\gamma}=\tfrac{d\,(p-1)\,\sigma_*\,M_*^{\gamma-1}}{d+2-p\,(d-2)}\,,\;\sigma_*(p):=\({\displaystyle 4}\,\tfrac{d+2-p\,(d-2)}{(p-1)^2\,(p+1)}\)^\frac{4\,p}{d-p\,(d-4)}.
\ee
Such a condition is not restrictive, as it is always possible to cover the general case by rescaling the inequality, but significantly simplifies the expressions. As we shall see in the proof, the only goal is to fix $\sigma=1$.

Our main result goes as follows.
%---------------------------------------------------------------
\begin{theorem}\label{Thm:GN} Let $d\ge 2$, $p>1$ and assume that $p<d/(d-2)$ if $d\ge3$. For any $f\in\L^{p+1}\cap\mathcal D^{1,2}(\R^d)$ such that Condition~\eqref{Normalization} holds, we have
\[
\ird{|\nabla f|^2}+\ird{|f|^{p+1}}-\K{p,d}\(\ird{|f|^{2\,p}}\)^\gamma\ge\mathsf C_{p,d}\,\frac{\(\mathcal R^{(p)}[f]\)^2}{\(\ird{\!|f|^{2\,p}}\)^\gamma}
\]
where $\gamma$ is given by \eqref{gamma}.\end{theorem}
%---------------------------------------------------------------
\noindent The constant $\mathsf C_{p,d}$ is positive and explicit. We do not know its optimal value. See Appendix~\ref{Sec:Appendix} for an expression of $\mathsf C_{p,d}$, which is such that
\[
\lim_{p\to1_+}\mathsf C_{p,d}=0\quad\mbox{and}\quad\lim_{p\to d/(d-2)_-}\mathsf C_{p,d}=0\;.
\]
The space $\L^{p+1}\cap\mathcal D^{1,2}(\R^d)$ is the natural space for Gagliardo-Nirenberg inequalities as it can be characterized as the completion of the space of smooth functions with compact support with respect to the norm $\|\cdot\|$ such that $\|f\|^2=\nrm{\nabla f}2^2+\nrm f{p+1}^2$. In this paper, we shall also use the notations $\nrm f{p,q}:=(\ird{|x|^p\,|f|^q})^{1/q}$, so that $\nrm fq=\nrm f{0,q}$.

\medskip Under Condition~\eqref{Normalization}, we shall deduce from Theorem~\ref{Thm:CK} that
\be{Ineq:CKp}
\mathcal R^{(p)}[f]\ge\mathsf C_{\rm CK}\,\nrm f{2\,p}^{2\,p\,(\gamma-2)}\,\inf_{g\in\mathfrak M_d^{(p)}}\nrm{|f|^{2\,p}-g^{2\,p}}1^2
\ee
with $\delta=d+2-p\,(d+6)$ for some constant $\mathsf C_{\rm CK}$ whose expression is given in Section~\ref{Sec:CK}, Eq.~\eqref{CK}. Putting this estimate together with the result of Theorem~\ref{Thm:GN}, with
\[
\mathfrak C_{p,d}:=\mathsf C_{d,p}\,{\mathsf C_{\rm CK}}^2\;,
\]
we obtain the following estimate.
%---------------------------------------------------------------
\begin{corollary}\label{Cor:GN} Under the same assumptions as in Theorem~\ref{Thm:GN}, we have
\begin{multline*}
\ird{|\nabla f|^2}+\ird{|f|^{p+1}}-\K{p,d}\(\ird{|f|^{2\,p}}\)^\gamma\\
\ge\mathfrak C_{p,d}\,\nrm f{{2\,p}}^{2\,p\,(\gamma-4)}\,\inf_{g\in\mathfrak M_d{(p)}}\nrm{|f|^{2\,p}-g^{2\,p}}1^4\;.
\end{multline*}
\end{corollary}
%---------------------------------------------------------------

The critical case $p=d/(d-2)$ corresponding to Sobolev's inequality raises a number of difficulties which are not under control at this stage. However, results which have been obtained in such a critical case, by different methods, are the main motivation for the present paper. 

In \cite[Question (c), p.~75]{MR790771}, H.~Brezis and E.~Lieb asked the question of what kind of distance to $\mathfrak M_d^{(p)}$ is controlled by the difference of the two terms in the critical Sobolev inequality written with an optimal constant. Some partial answers have been provided over the years, of which we can list the following ones. First G.~Bianchi and H.~Egnell gave in \cite{MR1124290} a result based on the concen\-tration-compactness method, which determines a non-constructive estimate for a distance to the set of optimal functions. In \cite{MR2538501}, A.~Cianchi, N.~Fusco, F.~Maggi and A.~Pratelli established an improved inequality using symmetrization methods. Also see~\cite{MR2508840} for an overview of various results based on such methods. Recently another type of improvement, which relates Sobolev's inequality to the Hardy-Little\-wood-Sobolev inequalities, has been established in \cite{1101}, based on the flow of a nonlinear diffusion equation, in the regime of extinction in finite time. Theorem~\ref{Thm:GN} does not provide an answer in the critical case, but gives an improvement with fully explicit constants in the subcritical regime. Our method of proof enlightens a new aspect of the problem. Indeed, Theorem~\ref{Thm:GN} shows that the difference of the two terms in the critical Sobolev inequality provides a better control under the additional information that $\nrm f{2,2\,p}$ is finite. Such a condition disappears in the setting of Corollary~\ref{Cor:GN}.

In this paper, our goal is to establish an improvement of Gagliardo-Nirenberg inequalities based on the flow of the \emph{fast diffusion equation} in the regime of convergence towards Barenblatt self-similar profiles, with an explicit measure of the distance to the set of optimal functions. Our approach is based on a \emph{relative entropy} functional. The method relies on a recent paper, \cite{1004}, which is itself based on a long series of studies on intermediate asymptotics of the fast diffusion equation, and on the entropy - entropy production method introduced in \cite{MR772092,AMTU} in the linear case and later extended to nonlinear diffusions: see \cite{MR760591,MR760592,MR1940370,MR1777035,MR1853037}. In this setting, having a finite second moment is crucial. Let us give some explanations.

\medskip Consider the fast diffusion equation with exponent $m$ given in terms of the exponent $p$ of Theorem~\ref{Thm:GN} by
\be{Eqn:pm}
p=\frac 1{2\,m-1}\quad\Longleftrightarrow\quad m=\frac{p+1}{2\,p}\;.
\ee
More specifically, for $m\in(0,1)$, we shall consider the solutions of
\be{Eqn2}
\frac{\partial u}{\partial t}+\nabla\cdot\left[u\,\(\eta\,\nabla u^{m-1}-2\,x\)\right]=0\quad t>0\;,\quad x\in\R^d
\ee
with initial datum $u(t=0,\cdot)=u_0$. Here $\eta$ is a positive parameter which does not depend on $t$. Let $u_\infty$ be the unique stationary solution such that $M=\ird u=\ird{u_\infty}$. It is given by
\[
u_\infty(x)=\(K+\frac 1\eta\,|x|^2\)^\frac 1{m-1}\quad\forall\;x\in\R^d
\]
for some positive constant $K$ which is uniquely determined by $M$. The following exponents are associated with the fast diffusion equation~\eqref{Eqn2} and will be used all over this paper:
\[
m_c:=\frac{d-2}d\;,\quad m_1:=\frac{d-1}d\quad\mbox{and}\quad\widetilde m_1:=\frac d{d+2}\;.
\]
To the critical exponent $2\,p=2\,d/(d-2)$ for Sobolev's inequality corresponds the critical exponent $m_1$ for the fast diffusion equation. For $d\ge 3$, the condition $p\in(1,d/(d-2))$ in Theorem~\ref{Thm:GN} is equivalent to $m\in(m_1,1)$ while for $d=2$, $p\in(1,\infty)$ means $m\in(1/2,1)$.

It has been established in \cite{MR760591,MR760592} that the \emph{relative entropy} (or \emph{free energy})
\[
\mathcal F[u|u_\infty]:=\frac 1{m-1}\ird{\big[u^m-u_\infty^m-m\,u_\infty^{m-1}\,(u-u_\infty)\big]}
\]
decays according~to
\[
\frac d{dt}\,\mathcal F[u(\cdot,t)|u_\infty]=-\mathcal I[u(\cdot,t)|u_\infty]
\]
if $u$ is a solution of \eqref{Eqn2}, where
\[
\mathcal I[u(\cdot,t)|u_\infty]:=\eta\,\frac m{1-m}\ird{u\,\left|\nabla u^{m-1}-\nabla u_\infty^{m-1}\right|^2}
\]
is the \emph{entropy production term} or \emph{relative Fisher information.} If $m\in[m_1,1)$, according to \cite{MR1940370}, these two functionals are related by a Gagliardo-Niren\-berg interpolation inequality, namely
\be{Ineq:E-EP}
\mathcal F[u|u_\infty]\le\frac 14\,\mathcal I[u|u_\infty]\;.
\ee
We shall give a concise proof of this inequality in the next section (see Remark~\ref{Rem:BEmethod}) based on the entropy - entropy production method, which amounts to relate $\frac d{dt}\,\mathcal I[u(\cdot,t)|u_\infty]$ and $\mathcal I[u(\cdot,t)|u_\infty]$. We shall later replace the diffusion parameter $\eta$ in~\eqref{Eqn2} by a time-dependent coefficient $\sigma(t)$, which is itself computed using the second moment of $u$, $\ird{|x|^2\,u(x,t)}$. By doing so, we will be able to capture the \emph{best matching} Barenblatt solution and get improved decay rates in the entropy - entropy production inequality. Elementary estimates allow to rephrase these improved rates into improved functional inequalities for $f$ such that $|f|^{2\,p}=u$, for any $p\in(1,d/(d-2))$, as in Theorem~\ref{Thm:GN}.

\medskip This paper is organized as follows. In Section~\ref{Sec:Bakry-Emery}, we apply the entropy - entropy production method to the fast diffusion equation as in \cite{MR1853037}. The key computation, without justifications for the integrations by parts, is reproduced here since we need it later in Section~\ref{Sec:Proofs}, in the case of a time-dependent diffusion coefficient. Next, in Section~\ref{Sec:CK}, we establish a new estimate of Csisz\'ar-Kullback type. By requiring a condition on the second moment, we are able to produce a new estimate which was not known before, namely to directly control the difference of the solution with a Barenblatt solution in~$\L^1(\R^d)$.

Second moment estimates are the key of a recent paper and we shall primarily refer to \cite{1004} in which the asymptotic behaviour of the solutions of the fast diffusion equation was studied. In Section~\ref{Sec:Matching} we recall the main results that were proved in \cite{1004}, and that are also needed in the present paper.

With these preliminaries in hand, an improved entropy - entropy production inequality is established in Section~\ref{Sec:ScaledEntropy-EP}, which is at the core of our paper. It is known since \cite{MR1940370} that entropy - entropy production inequalities amount to optimal Gagliardo-Nirenberg-Sobolev inequalities. Such a rephrasing of our result in a more standard form of functional inequalities is done in Section~\ref{Sec:Proofs}, which contains the proof of Theorem~\ref{Thm:GN}. Further observations have been collected in Section~\ref {Sec:Conclusion}. One of the striking results of our approach is that all constants can be explicitly computed. This is somewhat technical although not really difficult. To make the reading easier, explicit computations have been collected in Appendix~\ref{Sec:Appendix}.

%%%%%%%%%%%%%%%%%%%%%%%%%%%%%%%%%%%%%%%%%%%%%%%%%%%%%%%%%%%%%%%%%%%%%%
%%%%%%%%%%%%%%%%%%%%%%%%%%%%%%%%%%%%%%%%%%%%%%%%%%%%%%%%%%%%%%%%%%%%%%
\section{The entropy - entropy production method}\label{Sec:Bakry-Emery}

Consider a solution $u=u(x,t)$ of Eq.~\eqref{Eqn2} and define
\[
z(x,t):= \eta\,\nabla u^{m-1}-2\,x
\]
so that Eq.~\eqref{Eqn2} can be rewritten as
\[
\frac{\partial u}{\partial t}+\nabla\cdot\(u\, z\)=0\;.
\]
To keep notations compact, we shall use the following conventions. If $A=(A_{ij})_{i,j=1}^d$ and $B=(B_{ij})_{i,j=1}^d$ are two matrices, let $A:B=\sum_{i,j=1}^dA_{ij}\,B_{ij}$ and $|A|^2=A:A$. If $a$ and $b$ take values in $\R^d$, we adopt the definitions:
\[
a\cdot b=\sum_{i=1}^da_i\,b_i\;,\quad\nabla\cdot a=\sum_{i=1}^d\frac{\partial a_i} {\partial x_i}\;,\quad a\otimes b=(a_i\,b_j)_{i,j=1}^d\;,\quad\nabla\otimes a=\(\frac{\partial a_j}{\partial x_i}\)_{i,j=1}^d\;.
\]

Later we will need a version of the entropy - entropy production method in case of a time-dependent diffusion coefficient. Before doing so, let us recall the key computation of the standard method. With the above notations, it is straightforward to check that
\[
\frac{\partial z}{\partial t}=\eta\,(1-m)\nabla\big(u^{m-2}\,\nabla\cdot\(u\, z\)\big)\quad\mbox{and}\quad\nabla\otimes z=\eta\,\nabla\otimes\nabla u^{m-1}-2\,\mathrm{Id}\;.
\]
With these definitions, the time-derivative of $\frac{1-m}m\,\eta\,\mathcal I[u|u_\infty]=\ird{u\,|z|^2}$ can be computed as
\[
\frac d{dt}\ird{u\,|z|^2}=\ird{\frac{\partial u}{\partial t}\,|z|^2}+2\ird{u\,z\cdot\frac{\partial z}{\partial t}}\;.
\]
The first term can be evaluated by
\begin{eqnarray*}
&&\hspace*{-12pt}\ird{\frac{\partial u}{\partial t}\,|z|^2\kern-3pt}\\
&=&-\kern-2pt\ird{\nabla\cdot\(u\, z\)\,|z|^2\kern-3pt}\\
&=&2\kern-2pt\ird{u\,z\otimes z:\nabla\otimes z}\\
&=&2\,\eta\kern-2pt\ird{u\,z\otimes z:\nabla\otimes\nabla u^{m-1}}-\,4\kern-2pt\ird{u\,|z|^2\kern-3pt}\\
&=&2\,\eta\,(1-m)\kern-2pt\ird{u^{m-2}\,\nabla u\otimes\nabla:(u\,z\otimes z)}-\,4\kern-2pt\ird{u\,|z|^2\kern-3pt}\\
&=&2\,\eta\,(1-m)\kern-2pt\ird{u^{m-2}\(\nabla u\cdot z\)^2\kern-3pt}+\,2\,\eta\,(1-m)\kern-2pt\ird{u^{m-1}\(\nabla u\cdot z\)\(\nabla\cdot z\)}\\
&&+\,2\,\eta\,(1-m)\kern-2pt\ird{u^{m-1}\(z\otimes\nabla u\):\(\nabla\otimes z\)}-\,4\kern-2pt\ird{u\,|z|^2\kern-3pt}\;.
\end{eqnarray*}
The second term can be evaluated by
\begin{eqnarray*}
&&\hspace*{-12pt}2\kern-2pt\ird{u\,z\cdot\frac{\partial z}{\partial t}}\\
&=&2\,\eta\,(1-m)\kern -2pt\kern-2pt\ird{(u\,z\cdot\nabla)\(u^{m-2}\,\nabla\cdot(u\,z)\)}\\
&=&-2\,\eta\,(1-m)\kern -2pt\kern-2pt\ird{u^{m-2}\,\big(\nabla\cdot(u\,z)\big)^2\kern-3pt}\\
&=&-2\,\eta\,(1-m)\kern -2pt\kern-2pt\ird{\big[u^m(\nabla\cdot z)^2+2\,u^{m-1}\!\(\nabla u\cdot z\)\(\nabla\cdot z\)+u^{m-2}\!\(\nabla u\cdot z\)^2\big]}\;.
\end{eqnarray*}
Summarizing, we have found that
\begin{multline*}
\hspace*{-12pt}\ird{\frac{\partial u}{\partial t}\,|z|^2\kern-3pt}+4\ird{u\,|z|^2\kern-3pt}\\
=-2\,\eta\,(1-m)\kern -3pt\int_{\R^d}\kern -6pt u^{m-2}\big[u^2(\nabla\cdot z)^2+ u\(\nabla u\cdot z\)\(\nabla\cdot z\)
-u\,(z\otimes \nabla u):\!\(\nabla\!\otimes z\)\!\big]\;dx\;.
\end{multline*}
Using the fact that
\[
\frac{\partial^2z^j}{\partial x_i\,\partial x_j}=\frac{\partial^2z^i}{\partial x_j^2}\;,
\]
we obtain that
\begin{multline*}
\ird{u^{m-1}\,\(\nabla u\cdot z\)\,\(\nabla\cdot z\)}\\
=-\frac 1m\ird{u^m\,\(\nabla \cdot z\)^2}-\frac 1m\ird{u^m\,\sum_{i,j=1}^dz^i\,\frac{\partial^2z^j}{\partial x_i\,\partial x_j}}
\end{multline*}
and
\begin{multline*}
-\ird{u^{m-1}\,(z\otimes \nabla u):\!\(\nabla\!\otimes z\)}\\
=\frac 1m\ird{u^m\,|\nabla z|^2}+\frac 1m\ird{u^m\,\sum_{i,j=1}^dz^i\,\frac{\partial^2z^i}{\partial x_j^2}}
\end{multline*}
can be combined to give
\begin{multline*}
\ird{u^{m-2}\big[\,u\,\(\nabla u\cdot z\)\,\(\nabla\cdot z\)-u\,\nabla u\otimes z:\nabla\otimes z\big]}\\
=-\frac 1m\ird{u^m\,\(\nabla \cdot z\)^2}+\frac 1m\ird{u^m\,|\nabla z|^2}\;.
\end{multline*}
This shows that
\begin{multline}\label{Eqn:BakryEmery}
\frac d{dt}\ird{u\,|z|^2}+4\ird{u\,|z|^2}\\
=-2\,\eta\,\frac{1-m}m\ird{u^m\(|\nabla z|^2-(1-m)\,\(\nabla \cdot z\)^2\)}\;.
\end{multline}
By the arithmetic geometric inequality, we know that
\[
|\nabla z|^2-(1-m)\,\(\nabla \cdot z\)^2\ge0
\]
if $1-m\le1/d$, that is, if $m\ge m_1$. Altogether, we have formally established the following result.
%---------------------------------------------------------------
\begin{proposition}\label{Thm:BE} Let $d\ge 1$, $m\in(m_1,1)$ and assume that $u$ is a non-negative solution of \eqref{Eqn2} with initial datum $u_0$ in $\L^1(\R^d)$ such that $u_0^m$ and $x\mapsto |x|^2\,u_0$ are both integrable on $\R^d$. With the above defined notations, we get that
\[
\frac d{dt}\,\mathcal I[u(\cdot,t)|u_\infty]\le-\,4\,\mathcal I[u(\cdot,t)|u_\infty]\quad\forall\;t>0\;.
\]
\end{proposition}
%---------------------------------------------------------------
The proof of such a result requires to justify that all integrations by parts make sense. We refer to \cite{MR1777035,MR1986060} for a proof in the porous medium case ($m>1$) and to \cite{MR1853037} for $m_1\le m<1$. The case $m=1$ was covered long ago in \cite{MR772092}.

%---------------------------------------------------------------
\begin{remark}\label{Rem:BEmethod} Proposition~\ref{Thm:BE} provides a proof of~\eqref{Ineq:E-EP}. Indeed, with a Gronwall estimate, we first get that
\[
\mathcal I[u(\cdot,t)|u_\infty]\le \mathcal I[u_0|u_\infty]\,e^{-\,4\,t}\quad \forall\;t\ge0
\]
if $\mathcal I[u_0|u_\infty]$ is finite. Since $\mathcal I[u(\cdot,t)|u_\infty]$ is non-negative, we know that
\[
\lim_{t\to\infty}\mathcal I[u(\cdot,t)|u_\infty]=0\;,
\]
which proves the convergence of $u(\cdot,t)$ to $u_\infty$ as $t\to\infty$. As a consequence, we also have $\lim_{t\to\infty}\mathcal F[u(\cdot,t)|u_\infty]=0$ and since
\[
\frac d{dt}\big(\mathcal I[u(\cdot,t)|u_\infty]-4\,\mathcal F[u(\cdot,t)|u_\infty]\big)=\frac d{dt}\,\mathcal I[u(\cdot,t)|u_\infty]+\,4\,\mathcal I[u(\cdot,t)|u_\infty]\le0\;,
\]
an integration with respect to $t$ on $(0,\infty)$ shows that
\[
\mathcal I[u_0|u_\infty]-4\,\mathcal F [u_0|u_\infty]\ge0\;,
\]
which is precisely \eqref{Ineq:E-EP} written for $u=u_0$. \end{remark}
%---------------------------------------------------------------

%%%%%%%%%%%%%%%%%%%%%%%%%%%%%%%%%%%%%%%%%%%%%%%%%%%%%%%%%%%%%%%%%%%%%%
%%%%%%%%%%%%%%%%%%%%%%%%%%%%%%%%%%%%%%%%%%%%%%%%%%%%%%%%%%%%%%%%%%%%%%
\section{A Csisz\'ar-Kullback inequality}\label{Sec:CK}

Let $m\in(\widetilde m_1,1)$ with $\widetilde m_1=\frac d{d+2}$ and consider the relative entropy
\[
\mathcal F_\sigma[u]:=\frac 1{m-1}\int_{\R^d}\left[u^m-B_\sigma^m-m\,B_\sigma^{m-1}\,(u-B_\sigma)\right]\,dx
\]
for some Barenblatt function
\be{Eqn:Barenblatt}
B_\sigma(x):=\sigma^{-\frac d2}\(C_M+\tfrac 1\sigma\,|x|^2\)^\frac 1{m-1}\quad\forall\;x\in\R^d
\ee
where $\sigma$ is a positive constant and $C_M$ is chosen such that $\nrm{B_\sigma}1=M>0$. With $p$ and $m$ related by~\eqref{Eqn:pm}, the definition of $C_M$ coincides with the one of Section~\ref{Sec:Intro}. See details in Appendix~\ref{Sec:Appendix}.
%---------------------------------------------------------------
\begin{theorem}\label{Thm:CK} Let $d\ge 1$, $m\in(\widetilde m_1,1)$ and assume that $u$ is a non-negative function in $\L^1(\R^d)$ such that $u^m$ and $x\mapsto |x|^2\,u$ are both integrable on $\R^d$. If $\nrm u1=M$ and $\ird{|x|^2\,u}=\ird{|x|^2\,B_\sigma}$, then
\[
\frac{\mathcal F_\sigma[u]}{\sigma^{\frac d2(1-m)}}\ge \frac m{8\ird{B_1^m}}\(C_M\nrm{u-B_\sigma}1+\frac 1\sigma\,\ird{|x|^2\,|u-B_\sigma|}\)^2\,.
\]
\end{theorem}
%---------------------------------------------------------------
Notice that the condition $\ird{|x|^2\,u}=\ird{|x|^2\,B_\sigma}$ is explicit and determines $\sigma$ uniquely:
\[
\sigma=\frac 1{K_M}\ird{|x|^2\,u}\quad\mbox{with}\quad K_M:=\ird{|x|^2\,B_1}\;.
\]
For further details, see Lemma~\ref{Lem:ChoiceOfSigma} and \eqref{IPP} below, and Appendix~\ref{Sec:Appendix} for detailed expressions of $K_M$ and $\ird{B_1^m}$. With this choice of $\sigma$, since $B_\sigma^{m-1}=\sigma^{\frac d2(1-m)}\,C_M+\sigma^{\frac d2(m_c-m)}\,|x|^2$, we remark that $\ird{B_\sigma^{m-1}\,(u-B_\sigma)}=0$ so that the relative entropy reduces to
\[
\mathcal F_\sigma[u]:=\frac 1{m-1}\int_{\R^d}\left[u^m-B_\sigma^m\right]\,dx
\]

\begin{proof}[Proof of Theorem~\ref{Thm:CK}] Let $v:=u/B_\sigma$ and $d\mu_\sigma:=B_\sigma^m\,dx$. With these notations, we observe that
\begin{multline*}
\int_{\R^d}(v-1)\;d\mu_\sigma=\ird{B_\sigma^{m-1}\,(u-B_\sigma)}\\
=\sigma^{\frac d2\,(1-m)}\,C_M\ird{(u-B_\sigma)}+\sigma^{\frac d2(m_c-m)}\,\ird{|x|^2\,(u-B_\sigma)}=0\;.
\end{multline*}
Thus
\[
\int_{\R^d}(v-1)\;d\mu_\sigma=\int_{v>1}(v-1)\;d\mu_\sigma-\int_{v<1}(1- v)\;d\mu_\sigma=0\;,
\]
which, coupled with
\[
\int_{v>1}(v-1)\;d\mu_\sigma+\int_{v<1}(1- v)\;d\mu_\sigma=\int_{\R^d}|v-1|\;d\mu_\sigma\;,
\]
implies
\[
\ird{|u-B_\sigma|\,B_\sigma^{m-1}}=\int_{\R^d}|v-1|\;d\mu_\sigma=2\int_{v<1}\kern-1pt|v-1|\;d\mu_\sigma\;.
\]

On the other hand, a Taylor expansion shows that
\[
\mathcal F_\sigma[u]=\frac 1{m-1}\int_{\R^d}\big[v^m-1-m\,(v-1)\big]\,d\mu_\sigma=\frac m2\int_{\R^d}\xi^{m-2}\,|v-1|^2\,d\mu_\sigma
\]
for some function $\xi$ taking values in the interval $\(\min\{1,v\},\max\{1,v\}\)$, thus giving the lower bound
\[
\mathcal F_\sigma[u]\ge\frac m2\int_{v<1}\kern-1pt\xi^{m-2}\,|v-1|^2\,d\mu_\sigma\ge\frac m2\int_{v<1}\kern-1pt|v-1|^2\,d\mu_\sigma\;.
\]
Using the Cauchy-Schwarz inequality, we get
\[
\(\int_{v<1}\kern-1pt|v-1|\;d\mu_\sigma\)^2=\(\int_{v<1}\kern-1pt|v-1|\,B_\sigma^\frac m2\, B_\sigma^\frac m2\,dx\)^2\le\int_{v<1}\kern-1pt|v-1|^2\,d\mu_\sigma\ird{B_\sigma^m}
\]
and finally obtain that
\[
\mathcal F_\sigma[u]\ge\frac m2\,\frac{\(\int_{v<1}\kern-1pt|v-1|\;d\mu_\sigma\)^2}{\ird{B_\sigma^m}}=\frac m8\,\frac{\(\ird{|u-B_\sigma|\,B_\sigma^{m-1}}\)^2}{\ird{B_\sigma^m}}\;,
\]
which concludes the proof.
\end{proof}

Notice that the inequality of Theorem~\ref{Thm:CK} can be rewritten in terms of $|f|^{2\,p}=u$ and $g^{2\,p}=B_\sigma$ with $p=1/(2\,m-1)$. See Appendix~\ref{Sec:Appendix} for the computation of $\ird{B_\sigma^m}$, $\sigma$, $C_M$ and $K_M$ in terms of $\ird{|x|^2\,u}$ and $M_*$. In the framework of Corollary~\ref{Cor:GN}, we observe that Condition~\eqref{Normalization} can be rephrased~as
\be{Sigma*}
\sigma=\frac 1{K_M}\nrm f{2,2\,p}^{2\,p}=\frac 1{K_1}\frac{\nrm f{2,2\,p}^{2\,p}}{\nrm f{2\,p}^{2\,p\,\gamma}}=\sigma_*\;.
\ee
Altogether we find in such a case that
\[
\mathcal R^{(p)}[f]=\tfrac{p-1}{p+1}\,\mathcal F_{\sigma_*}[u]\ge\mathsf C_{\rm CK}\,\nrm{\,|f|^{2\,p}-|g|^{2\,p}\,}1^2
\]
with
\be{CK}
\mathsf C_{\rm CK}=\tfrac{p-1}{p+1}\,\tfrac{d+2-p\,(d-2)}{32\,p}\,\sigma_*^{d\,\frac{p-1}{4\,p}}\,M_*^{1-\gamma}\,.
\ee

%---------------------------------------------------------------
\begin{remark} Various other estimates can be derived, based on second order Taylor expansions. For instance, as in \cite{MR1940370}, we can write that
\[
\mathcal F_\sigma[u]=\int_{\R^d}\big[\psi(v^m)-\psi(1)-\psi'(1)\,(v^m-1)\big]\,d\mu_\sigma
\]
with $v:=u/B_\sigma$ and $\psi(s):=\frac m{1-m}\,s^{1/m}$, and get
\[
\mathcal F_\sigma[u]\ge\frac 1m\,2^{-2\,m}\,\frac{\|v^m-1\|_{\L^{1/m}(\R^d,d\mu_\sigma)}^2}{\max\big\{\|v^m\|_{\L^{1/m}(\R^d,d\mu_\sigma)},\|1\|_{\L^{1/m}(\R^d,d\mu_\sigma)}\big\}^{2-\frac 1m}}\;.
\]
Using $\|v^m\|_{\L^{1/m}(\R^d,d\mu_\sigma)}=\|1\|_{\L^{1/m}(\R^d,d\mu_\sigma)}=\nrm{B_\sigma^m}1^m$ and
\begin{multline*}
\ird{|u^m-B_\sigma^m|}=\ird{|u^m-B_\sigma^m|\,B_\sigma^{m\,(m-1)}\,B_\sigma^{m\,(1-m)}}\\
\le\|v^m-1\|_{\L^{1/m}(\R^d,d\mu_\sigma)}\,\nrm{B_\sigma^m}1^{1-m}
\end{multline*}
by the Cauchy-Schwarz inequality, we find
\[
\mathcal F_\sigma[u]\ge\frac{\nrm{u^m-B_\sigma^m}1^2}{m\,2^{\,2\,m}\,\nrm{B_\sigma^m}1}\;.
\]
With $f=u^{m-\frac 12}$, this also gives another estimate of Csisz\'ar-Kullback type, namely
\[
\mathcal R^{(p)}[f]\ge\frac{\kappa_{p,d}}{\nrm f{2,2\,p}^{\frac d2\,(p-1)}\,\nrm f{2\,p}^{\frac 12\,(d+2-p\,(d-2))}}\,\inf_{g\in\mathfrak M_d^{(p)}}\nrm{|f|^{p+1}-g^{p+1}}1^2\;,
\]
for some positive constant $\kappa_{p,d}$, which is valid for any $p\in(1,\infty)$ if $d=2$ and any $p\in(1,\frac d{d-2}]$ if $d\ge3$. Also see \cite{MR1801751,MR1777035,MR1951784,DoKa06} for further results on Csisz\'ar-Kullback type inequalities corresponding to entropies associated with porous media and fast diffusion equations. \end{remark}
%---------------------------------------------------------------

%%%%%%%%%%%%%%%%%%%%%%%%%%%%%%%%%%%%%%%%%%%%%%%%%%%%%%%%%%%%%%%%%%%%%%
%%%%%%%%%%%%%%%%%%%%%%%%%%%%%%%%%%%%%%%%%%%%%%%%%%%%%%%%%%%%%%%%%%%%%%
\section{Recent results on the optimal matching by Barenblatt solutions}\label{Sec:Matching}

Consider on $\R^d$ the fast diffusion equation with harmonic confining potential given by
\be{Eqn2bis}
\frac{\partial u}{\partial t}+\nabla\cdot\left[u\,\(\sigma^{\frac d2(m-m_c)} \,\nabla u^{m-1}-2\,x\)\right]=0\quad t>0\;,\quad x\in\R^d\,,
\ee
with initial datum $u_0$. Here $\sigma$ is a function of $t$. Let us summarize some results obtained in \cite{1004} and the strategy of their proofs.

\medskip\noindent\emph{Result 1.} At any time $t>0$, we can choose the \emph{best matching Barenblatt} as follows. Consider a given function $u$ and optimize $\lambda\mapsto\mathcal F_\lambda[u]$.
%---------------------------------------------------------------
\begin{lemma}\label{Lem:ChoiceOfSigma} For any given $u\in\L^1_+(\R^d)$ such that $u^m$ and $|x|^2\,u$ are both integrable, if $m\in(\widetilde m_1,1)$, there is a unique $\lambda=\lambda^*>0$ which minimizes $\lambda\mapsto\mathcal F_\lambda[u]$, and it is explicitly given by
\[\label{min}
{\lambda^*}=\frac 1{K_M}\,\ird{|x|^2\,u}
\]
where $K_M=\ird{|x|^2\,B_1}$. For $\lambda=\lambda^*$, the Barenblatt profile $B_\lambda$ satisfies
\[\label{equal-moment}
\ird{|x|^2\,B_\lambda}=\ird{|x|^2\,u}\;.
\]
\end{lemma}
%---------------------------------------------------------------
\noindent As a consequence, we know that
\[
\frac d{d\lambda}\big(\mathcal F_\lambda[u]\big)_{\lambda=\lambda_*}=0\;.
\]
Of course, if $u$ is a solution of \eqref{Eqn2bis}, the value of $\lambda$ in Lemma~\ref{Lem:ChoiceOfSigma} may depend on~$t$. Now we choose $\sigma(t)=\lambda(t)$, \emph{i.e.,}
\be{Eqn:sigma}
\sigma(t)=\frac 1{K_M}\,\ird{|x|^2\,u(x,t)}\quad\forall\;t\ge 0\;.
\ee
This makes \eqref{Eqn2bis} a non-local equation.

\medskip\noindent\emph{Result 2.} With the above choice, if we consider a solution of~\eqref{Eqn2bis} and compute the time derivative of the relative entropy, we find that
\[\label{Eqn:TwoTerms}
\frac d{dt}\,\mathcal F_{\sigma(t)}[u(\cdot,t)]=\sigma'(t)\(\frac d{d\sigma}\,\mathcal F_\sigma[u]\)_{|\sigma=\sigma(t)}\kern -5pt+\frac m{m-1}\ird{\(u^{m-1}-B_{\sigma(t)}^{m-1}\)\frac{\partial u}{\partial t}}\;.
\]
However, as a consequence of the choice~\eqref{Eqn:sigma} and of Lemma~\ref{Lem:ChoiceOfSigma}, we know that
\[
\(\frac d{d\sigma}\,\mathcal F_\sigma[u]\)_{|\sigma=\sigma(t)}=0\;,
\]
and finally obtain
\be{Eqn:Standard}
\frac d{dt}\,\mathcal F_{\sigma(t)}[u(\cdot,t)]=-\frac{m\,\sigma(t)^{\frac d2(m-m_c)}} {1-m}\ird{u\left|\nabla\left[u^{m-1}-B_{\sigma(t)}^{m-1}\right]\right|^2}\;.
\ee
The computation then goes as in \cite{BBDGV,BDGV} (also see \cite{MR760591,MR760592,MR1940370} for details). With our choice of $\sigma$, we gain an additional orthogonality condition which is useful for improving the rates of convergence (see~\cite [Theorem~1]{1004}) in the asymptotic regime $t\to\infty$, compared to the results of \cite{BDGV} (also see below).

\medskip\noindent\emph{Result 3.} Now let us state one more result of \cite{1004} which is of interest for the present paper.
%---------------------------------------------------------------
\begin{lemma}\label{Lem:Asymptotic} With the above notations, if $u$ and $\sigma$ are defined respectively by~\eqref{Eqn2bis} and \eqref{Eqn:sigma}, then the function $t\mapsto\sigma(t)$ is positive, decreasing, with $\sigma_\infty:=\lim_{t\to\infty}\sigma(t)>0$ and
\be{ODE:sigma}
\sigma'(t)=-2\,d\,\frac{(1-m)^2}{m\,K_M}\,\sigma^{\frac d2(m-m_c)}\,\mathcal F_{\sigma(t)}[u(\cdot,t)]\le 0\;.
\ee
\end{lemma}
%---------------------------------------------------------------

The main difficulty is to establish that $\sigma_\infty$ is positive. This can be done with an appropriate change of variables which reduces \eqref{Eqn2bis} to the case where~$\sigma$ does not depend on $t$. In \cite{1004}, a proof has been given, based on asymptotic results for the fast diffusion equation that were established in \cite{MR1940370,BBDGV,BBDGV-CRAS,BDGV}. An alternative proof will be given in Remark~\ref{Rem:SigmaInfty}, below.

%%%%%%%%%%%%%%%%%%%%%%%%%%%%%%%%%%%%%%%%%%%%%%%%%%%%%%%%%%%%%%%%%%%%%%
%%%%%%%%%%%%%%%%%%%%%%%%%%%%%%%%%%%%%%%%%%%%%%%%%%%%%%%%%%%%%%%%%%%%%%
\section{The scaled entropy - entropy production inequality}\label{Sec:ScaledEntropy-EP}

Consider the relative Fisher information
\[
\mathcal I_\sigma[u]:=\sigma^{\frac d2(m-m_c)}\,\frac m{1-m}\ird{u\,\left|\nabla u^{m-1}-\nabla B_\sigma^{m-1}\right|^2}\;.
\]
By applying \eqref{Ineq:E-EP} with $u_\infty=B_1$ and $\eta=1$ to $x\mapsto\sigma^{d/2}\,u(\sqrt\sigma\,x)$ and using the fact that $B_1(x)=\sigma^{d/2}\,B_\sigma(\sqrt\sigma\,x)$, we get the inequality
\[\label{Ineq:E-EPrescaled}
\mathcal F_\sigma[u]\le\frac 14\,\mathcal I_\sigma[u]\;.
\]
Now, if $\sigma$ is time-dependent as in Section~\ref{Sec:Matching}, we have the following relations.
%---------------------------------------------------------------
\begin{lemma}\label{Lem:BE-rescaled} If $u$ is a solution of \eqref{Eqn2bis} with $\sigma(t)=\frac 1{K_M}\,\ird{|x|^2\,u(x,t)}$, then $\sigma$ satisfies \eqref{ODE:sigma}. Moreover, for any $t\ge0$, we have
\be{Eqn:entr0}
\frac d{dt}\,\mathcal F_{\sigma(t)}[u(\cdot,t)]=-\mathcal I_{\sigma(t)}[u(\cdot,t)]
\ee
and
\be{Eqn:prod0}
\frac d{dt}\,\mathcal I_{\sigma(t)}[u(\cdot,t)] \le-\left[ 4+\frac 12\,(m-m_c)\,(m-m_1)\,d^2\,\frac{|\sigma'(t)|}{\sigma(t)} \,\right]\mathcal I_{\sigma(t)}[u(\cdot,t)]\;.
\ee
\end{lemma}
%---------------------------------------------------------------
\begin{proof} Eq.~\eqref{ODE:sigma} and~\eqref{Eqn:entr0} have already been stated respectively in Lemma~\ref{Lem:Asymptotic} and in~\eqref{Eqn:Standard}. They are recalled here only for the convenience of the reader. It remains to prove~\eqref{Eqn:prod0}.

For any given $\sigma=\sigma(t)$, Proposition~\ref{Thm:BE} gives
\begin{multline*}
\frac d{dt}\,\mathcal I_{\sigma(t)}[u(\cdot,t)]= \(\frac d{dt}\,\mathcal I_{\lambda}[u(\cdot,t)]\)_{|\lambda =\sigma(t)}+\sigma'(t)\(\frac d{d\lambda}\,\mathcal I_\lambda[u]\)_{|\lambda=\sigma(t)}\\
\le -\,4\,\mathcal I_{\sigma(t)}[u(\cdot,t)]+\sigma'(t)\(\frac d{d\lambda}\,\mathcal I_\lambda[u]\)_{|\lambda=\sigma(t)}\,.
\end{multline*}
Owing to the definition of $\mathcal I_\lambda$, we obtain
\begin{multline*}
\frac d{d\lambda}\,\mathcal I_\lambda[u]=\frac d2\,(m-m_c)\,\frac 1\lambda\,\mathcal I_\lambda[u]\\
-\frac m{1-m}\,\lambda^{\frac d2(m-m_c)}\ird{\!2\,u\,\(\nabla u^{m-1}-\nabla B_\lambda^{m-1}\)\cdot \frac d{d\lambda} \(\nabla B_\lambda^{m-1}\) }\;.
\end{multline*}
By definition \eqref{Eqn:Barenblatt}, $\nabla B_\lambda^{m-1}(x)=2\,x\, \lambda^{-\frac d2(m-m_c)}$, which implies
\[
\lambda^{\frac d2(m-m_c)}\,\frac{d}{d \lambda} \(\nabla B_\lambda^{m-1}\)=-\frac d\lambda\,(m-m_c)\,x\;.
\]
Substituting this expression into the above computation and integrating by parts, we conclude with the equality
\begin{multline*}
\frac d{d\lambda}\,\mathcal I_\lambda[u]= \frac d2\,(m-m_c)\,\frac 1\lambda\,\mathcal I_\lambda[u]\\
-\frac {2\,d}\lambda\,(m-m_c) \left[\frac{2\,m\,\lambda^{-\frac d2(m-m_c)}}{1-m}\ird{|x|^2\,u}-d\ird{ u^m }\right]\,.
\end{multline*}
A simple computation shows that
\be{IPP}
d\ird{B_1^m}=-\ird{x\cdot\nabla B_1^m}=\frac{2\,m}{1-m}\ird{|x|^2\,B_1}=\frac{2\,m}{1-m}\,K_M
\ee
and, as a consequence, if $\lambda=\sigma=\frac 1{K_M}\,\ird{|x|^2\,u}$, then
\[
\frac{2\,m\,\lambda^{-\frac d2(m-m_c)}}{1-m}\ird{|x|^2\,u}=d\ird{B_\lambda^m}\;,
\]
and finally
\be{Eqn:Ineq:FisherImprovedLambda}
\frac d{d\lambda}\,\mathcal I_\lambda[u]= \frac d{2\,\lambda}\,(m-m_c)\Big(\mathcal I_\lambda[u]-\,4\,d\,(1-m)\,\mathcal F_\lambda[u]\Big)\,.
\ee
Altogether, we have found that
\begin{multline*}\label{Ineq:FisherImproved}
\frac d{dt}\,\Big(\mathcal I_{\sigma(t)}[u(\cdot,t)]\Big)+\,4\,\mathcal I_{\sigma(t)}[u(\cdot,t)]\\
\le\frac d2\,(m-m_c)\,\frac{\sigma'(t)}{\sigma(t)}\Big(\mathcal I_{\sigma(t)}[u]-\,4\,d\,(1-m)\,\mathcal F_{\sigma(t)}[u]\Big)\,.
\end{multline*}
The last term of the right hand side is non-positive because by \eqref{ODE:sigma} we know that $\sigma'(t) \le 0$ and
\begin{multline*}
\mathcal I_{\sigma(t)}[u]-\,4\,d\,(1-m)\,\mathcal F_{\sigma(t)}[u]\\
=d\,(1-m)\Big(\mathcal I_{\sigma(t)}[u]-\,4\,\mathcal F_{\sigma(t)}[u]\Big)+\,d\,(m-m_1)\,\mathcal I_{\sigma(t)}[u]\\
\ge\,d\,(m-m_1)\,\mathcal I_{\sigma(t)}[u]\ge0\;.
\end{multline*}
This implies \eqref{Eqn:prod0}.\end{proof}

To avoid carrying heavy notations, let us write
\[
f(t):=\mathcal F_{\sigma(t)}[u(\cdot,t)]\quad\mbox{and}\quad j(t):=\mathcal J_{\sigma(t)}[u(\cdot,t)]
\]
and denote $f(0)$, $j(0)$ and $\sigma(0)$ respectively by $f_0$, $j_0$ and $\sigma_0$. Estimates~\eqref{ODE:sigma}, \eqref{Eqn:entr0} and \eqref{Eqn:prod0} can be rewritten as
\be{System}
\left\{\begin{array}{l}
f'=-j\le 0\\[4pt]
\sigma'=-\kappa_1\,\sigma^{\frac d2(m-m_c)}\,f\le0\\[4pt]
j'+\,4\,j\le\kappa_2\,j\,\frac{\sigma'}\sigma
\end{array}\right.
\ee
where the constants $\kappa_i$, $i=1$, $2$, are given by
\[
\kappa_1:=2\,d\,\frac{(1-m)^2}{m\,K_M}\quad\mbox{and}\quad\kappa_2:=\frac 12\,(m-m_c)\,(m-m_1)\,d^2\,.
\]
Using the fact that $\lim_{t\to\infty}f(t)=\lim_{t\to\infty}f(t)=0$, as in the proof of Proposition~\ref{Thm:BE}, we find that $j(t)-\,4\,f(t)\ge 0$ and $f(t)\le f_0\,e^{-\,4\,t}$ for any $t\ge 0$.
%---------------------------------------------------------------
\begin{remark}\label{Rem:SigmaInfty} 
The decay of $\sigma$ can be estimated by
\[
-\frac d{dt}\(\sigma^{\frac d2\,(1-m)}\)=\frac d2\,(1-m)\,\kappa_1\,f\le\frac d2\,(1-m)\,\kappa_1\,f_0\,e^{-\,4\,t}\,,
\]
thus showing that $\sigma_\infty^{\frac d2\,(1-m)}\ge\sigma_0^{\frac d2\,(1-m)}-\frac{d^2\,(1-m)^3}{4\,m\,K_M}\,f_0$. Since $u_0$ and $B_{\sigma_0}$ have the same mass and second moment, we know that 
\[
f_0=\frac 1{1-m}\ird{\(B_{\sigma_0}^m-u_0^m\)}=\frac{2\,m\,K_M}{d\,(1-m)^2}\,\sigma_0^{\frac d2(1-m)}-\frac 1{(1-m)}\, \ird{u_0^m}\;.
\]
Hence we end up with the positive lower bound
\[\label{SpecificLowerBound}
\sigma_\infty^{\frac d2(1-m) }\ge\frac d2\,(m-m_c)\,\sigma_0^{\frac d2(1-m)}+\frac{d^2\,(1-m)^2}{4\,m\,K_M}\ird{u_0^m}\,.
\]
\end{remark}
%---------------------------------------------------------------
{}From~\eqref{System} we get the estimates $\sigma(t)\le\sigma_0$ for any $t\ge0$ and
\[
j'-\,4\,f'=j'+\,4\,j\le\kappa_2\,j\,\frac{\sigma'}\sigma=\kappa_1\,\kappa_2\,\sigma^{-\frac d2\,(1-m)}\,f\,f'\le\kappa_1\,\kappa_2\,\sigma_0^{-\frac d2\,(1-m)}\,f\,f'
\]
Integrating from $0$ to $\infty$ with respect to~$t$ and taking into account the fact that $\lim_{t\to\infty}f(t)=\lim_{t\to\infty}f(t)=0$, we get
\[
-\,j_0+\,4\,f_0\le-\,\frac 12\,\kappa_1\,\kappa_2\,\sigma_0^{-\frac d2\,(1-m)}\,f_0^2\;.
\]
By rewriting this estimate in terms of $\mathcal F_{\sigma_0}[u_0]=f_0$, $\mathcal I_{\sigma_0}[u_0]=j_0$ and after omitting the index $0$, we have achieved our key estimate, which can be written using
\[\label{Cmd}
C_{m,d}:=\frac{d^3}{2\,m\,K_M}\,(m-m_c)\,(m-m_1)\,(1-m)^2
\]
as follows.
%---------------------------------------------------------------
\begin{theorem}\label{Thm:IN} Let $d\ge 1$, $m\in(m_1,1)$ and assume that $u$ is a non-negative function in $\L^1(\R^d)$ such that $u^m$ and $x\mapsto |x|^2\,u$ are both integrable on $\R^d$. Let $\sigma=\frac 1{K_M}\,\ird{|x|^2\,u(x)}$ where $M=\ird{u(x)}$. Then the following inequality holds
\be{Eqn:improved}
4\,\mathcal F_\sigma[u]+C_{m,d}\,\frac{\(\mathcal F_\sigma[u]\)^2}{\sigma^{\frac d2\,(1-m)}}\le\mathcal I_\sigma[u]\;.
\ee
\end{theorem}
%---------------------------------------------------------------
\noindent Recall that $K_M=K_1\,M^\gamma$, with $\gamma=\frac{(d+2)\,m-d}{d\,(m-m_c)}$. See Appendix~\ref{Sec:Appendix} for details. Notice that this definition of $\gamma$ is compatible with the one of Theorem~\ref{Thm:GN} if $p=1/(2\,m-1)$.

%---------------------------------------------------------------
\begin{remark}\label{Rem:Better} If we do not drop any term in the proof of Proposition~\ref{Thm:BE} and Lemma~\ref{Lem:BE-rescaled}, an ODE can be obtained for $j$, based on \eqref{Eqn:BakryEmery} and \eqref{Eqn:Ineq:FisherImprovedLambda} and we can replace \eqref{System} by a system of coupled ODEs that reads
\[
\left\{\begin{array}{l}
f'=-j\le 0\\[4pt]
\sigma'=-\kappa_1\,\sigma^{\frac d2(m-m_c)}\,f\le0\\[4pt]
j'+\,4\,j=d\,(1-m)\,(j-4\,f)\,\frac{\sigma'}\sigma+\kappa_2\,j\,\frac{\sigma'}\sigma\,-\mathsf r
\end{array}\right.
\]
where $\mathsf r:=2\,\sigma^{d\,(m-m_c)}\ird{u^m\(|\nabla z|^2-(1-m)\,\(\nabla \cdot z\)^2\)}\ge 0$ and $z:=\sigma^{\frac d2(m-m_c)}\,\nabla u^{m-1}-2\,x$.

It is then clear that the estimates $\sigma\le\sigma_0$ and $j'+\,4\,j\le\kappa_2\,j\,\frac{\sigma'}\sigma$, which have been used for the proof of Theorem~\ref{Thm:IN}, are not optimal. \end{remark}
%---------------------------------------------------------------

%%%%%%%%%%%%%%%%%%%%%%%%%%%%%%%%%%%%%%%%%%%%%%%%%%%%%%%%%%%%%%%%%%%%%%
%%%%%%%%%%%%%%%%%%%%%%%%%%%%%%%%%%%%%%%%%%%%%%%%%%%%%%%%%%%%%%%%%%%%%%
\section{Proofs of Theorem~\ref{Thm:GN} and Corollary~\ref{Cor:GN}}\label{Sec:Proofs}

Let us start by rephrasing Theorem~\ref{Thm:IN} in terms of $f=u^{m-1/2}$. Assume that
\[
M=\ird u=\ird{|f|^{2\,p}}\quad\mbox{and}\quad\sigma=\frac 1{K_M}\ird{|x|^2\,u}=\ird{|x|^2\,|f|^{2\,p}}
\]
where $p=1/(2\,m-1)$ and using the notation $f_{M,0,\sigma}^{(p)}\in\mathfrak M_d^{(p)}$ defined in Section~\ref{Sec:Intro}, consider the functional
\[
\mathsf R^{(p)}[f]:=-\tfrac{2\,p}{p+1} \ird{\left[(|f|^{p+1}-\big(f_{M,0,\sigma}^{(p)}\big)^{p+1}\right]}\;.
\]
In preparation for the proof of Theorem~\ref{Thm:GN}, we can state the following result.
%---------------------------------------------------------------
\begin{corollary}\label{Cor:11} Let $d\ge 2$, $p>1$ and assume that $p<d/(d-2)$ if $d\ge3$. For any $f\in\L^{p+1}\cap\mathcal D^{1,2}(\R^d)$ such that Condition~\eqref{Normalization} holds, we have
\[
\ird{|\nabla f|^2}+\ird{|f|^{p+1}}-\K{p,d}\(\ird{|f|^{2\,p}}\)^\gamma\ge\mathsf C_{p,d}\,\frac{\(\mathsf R^{(p)}[f]\)^2}{\(\ird{\!|f|^{2\,p}}\)^\gamma}
\]
where $\gamma$ is given by \eqref{gamma}.\end{corollary}
%---------------------------------------------------------------
This results is slightly more precise than the one given in Theorem~\ref{Thm:GN}, as we simply measure the distance to a special function in $\mathfrak M_d^{(p)}$, the one with same mass and second moment, centered at $0$. The constant $\mathsf C_{p,d}$ is the same as in Theorem~\ref{Thm:GN}: see Appendix~\ref{Sec:Appendix} for its expression.

\begin{proof} By expanding the square in $\mathcal I_\sigma[u]$ and collecting the terms with the ones of $\mathcal F_\sigma[u]$, we find that
\begin{multline*}
\frac 14\,\mathcal I_\sigma[u]-\mathcal F_\sigma[u]=\tfrac{m\,(1-m)}{(2\,m-1)^2}\,\sigma^{\frac d2(m-m_c)}\ird{|\nabla u^{m-\frac 12}|^2}\\
+d\,\tfrac{m-m_1}{1-m}\ird{u^m}+\tfrac 1{1-m}\(m\,K_M\,\sigma^{\frac d2\,(1-m)}-\ird{B_\sigma^m}\)\,.
\end{multline*}
The last term of the right hand side can be rewritten as
\[
\tfrac 1{1-m}\(m\,K_M\,\sigma^{\frac d2\,(1-m)}-\ird{B_\sigma^m}\)=-\tfrac m{1-m}\,\tfrac{d\,(m-m_c)}{(d+2)\,m-d}\,\sigma^{\frac d2\,(1-m)}\,C_1\,M^\gamma
\]
with $\gamma=\frac{(d+2)\,m-d}{d\,(m-m_c)}$ (as in the previous Section) and $C_1=M_*^{1-\gamma}$ (see Appendix~\ref{Sec:Appendix} for details). Consequently Inequality \eqref{Eqn:improved} can be equivalently rewritten~as
\begin{multline}\label{PreCpd}
\tfrac{m\,(1-m)}{(2\,m-1)^2}\,\sigma^{\frac d2(m-m_c)}\ird{|\nabla u^{m-\frac 12}|^2}+d\,\tfrac{m-m_1}{1-m}\ird{u^m}\\
\ge\tfrac m{1-m}\,\tfrac{d\,(m-m_c)}{(d+2)\,m-d}\,\sigma^{\frac d2\,(1-m)}\,C_1\,M^\gamma\\
+\tfrac{d^3\,(m-m_c)\,(m-m_1)\,(1-m)^2}{8\,m\,K_1}\,\frac{\(\mathcal F_\sigma[u]\)^2}{M^\gamma\,\sigma^{\frac d2\,(1-m)}}\;.
\end{multline}
This inequality is invariant under scaling and homogeneous. As already noticed in \eqref{Sigma*}, Condition~\eqref{Normalization} means $\sigma=\sigma_*$, that is $\tfrac{m\,(1-m)}{(2\,m-1)^2}\,\sigma^{\frac d2(m-m_c)}=d\,\tfrac{m-m_1}{1-m}$. Using the explicit expressions that can be found in Appendix~\ref{Sec:Appendix} and reexpressing all quantities in terms of $p=\frac 1{2\,m-1}$ completes the proof of Corollary~\ref{Cor:11}. See Appendix~\ref{Sec:Appendix} for an expression of $\mathsf C_{p,d}$.
\end{proof}

\begin{proof}[Proof of Theorem~\ref{Thm:GN}] It is itself a simple consequences of Corollary~\ref{Cor:11}.

Let us consider the relative entropy with respect to a general Barenblatt function, not even necessarilynormalized with respect to its mass. For a given function $u\in\L^1_+(\R^d)$ with $u^m\in\L^1(\R^d)$ and $|x|^2\,u\in\L^1(\R^d)$, we can consider on $(0,\infty)\times\R^d\times(0,\infty)$ the function $h$ defined by
\[
h(C,y,\sigma)=\frac
1{m-1}\ird{\big[u^m-B_{C,y,\sigma}^m-m\,B_{C,y,\sigma}^{m-1}\,(u-B_{C,y,\sigma})\big]}
\]
where $B_{C,y,\sigma}$ is a general Barenblatt function
\[
B_{C,y,\sigma}(x):=\sigma^{-\frac d2}\(C+\tfrac 1\sigma\,|x-y|^2\)^\frac1{m-1}\quad\forall\;x\in\R^d\,.
\]
An elementary computation shows that
\begin{eqnarray*}
&&\frac{\partial h}{\partial C}=\frac{m\,\sigma^{\frac d2\,(1-m)}}{1-m}\ird{\(u-B_{C,y,\sigma}\)}\;,\\
&&\nabla_yh=\frac{2\,m\,\sigma^{-\frac d2\,(m-m_c)}}{1-m}\ird{(x-y)\(u-B_{C,y,\sigma}\)}\;,\\
&&\frac{\partial h}{\partial\sigma}=m\,\frac d2\,\sigma^{-\frac d2\,(m-m_c)}\Bigg[\,C\ird{\(u-B_{C,y,\sigma}\)}\\
&&\hspace*{4cm}-\frac{m-m_c}{1-m}\,\frac 1\sigma\ird{|x-y|^2\(u-B_{C,y,\sigma}\)}\Bigg]\;.
\end{eqnarray*}
Optimizing with respect to $C$ fixes $C=C_M$, with $M=\ird u$. Once $C=C_M$ is assumed, optimizing with respect to $\sigma$ amounts to choose it such that $\ird{|x|^2\,B_{C,y,\sigma}}=\ird{|x-y|^2\,u}$ as it has been shown in Lemma~\ref {Lem:ChoiceOfSigma}.

This completes the proof of Theorem~\ref{Thm:GN}, since $\mathcal R^{(p)}[f]\ge\mathsf R^{(p)}[f]$ by definition of $\mathcal R^{(p)}$ (see Section~\ref{Sec:Intro}). Notice that optimizing on $y$ amounts to fix the center of mass of the Barenblatt function to be the same as the one of $u$. This is however required neither in the proof of Corollary~\ref{Cor:11} nor in the one of Theorem~\ref{Thm:GN}.\end{proof}

\begin{proof}[Proof of Corollary~\ref{Cor:GN}] It is a straightforward consequence of Theorem~\ref{Thm:GN} and of the Csisz\'ar-Kullback inequality~\eqref{Ineq:CKp} when $f\in\mathcal D^{1,2}(\R^d)$ is such that $\nrm f{2,2\,p}$ is finite. However, $\nrm f{2,2\,p}$ does not enter in the inequality. Since smooth functions with compact support (for which $\nrm f{2,2\,p}$ is obviously finite) are dense $\mathcal D^{1,2}(\R^d)$, the inequality therefore holds without restriction, by density.\end{proof}

%%%%%%%%%%%%%%%%%%%%%%%%%%%%%%%%%%%%%%%%%%%%%%%%%%%%%%%%%%%%%%%%%%%%%%
%%%%%%%%%%%%%%%%%%%%%%%%%%%%%%%%%%%%%%%%%%%%%%%%%%%%%%%%%%%%%%%%%%%%%%
\section{Concluding remarks}\label{Sec:Conclusion}

Let us conclude this paper with a few remarks. First of all, notice that Theorem~\ref{Thm:CK} gives a stronger information than Theorem~\ref{Thm:GN}, as not only the $\L^1(\R^d,dx)$ norm is controlled, but also a stronger norm involving the second moment, properly scaled.

\medskip No condition is imposed on the location of the center of mass, which simply has to satisfy $\(\ird{x\,u}\)^2\le \ird u\ird{|x|^2\,u}=\sigma\,M\,K_M$ according to the Cauchy-Schwarz inequality. Hence in the definition of $\mathcal R[f]$ and $\mathcal R^{(p)}[f]$ (in Theorem~\ref{Thm:GN}) as well as in Corollary~\ref{Cor:GN}, the result holds without optimizing on $y\in\R^d$. In~\cite{BDGV,1004}, improved asymptotic rates were obtained by fixing the center of mass in order to kill the linear mode associated to the translation invariance of the Barenblatt functions. Here this is not required since, as $t\to\infty$, the squared relative entropy is simply of higher order. Our improvement is better when the relative entropy is large, and is clearly not optimal for large values~of~$t$.

\medskip Our approach differs from the one of G.~Bianchi and H.~Egnell in \cite{MR1124290} and the one of A.~Cianchi, N.~Fusco, F.~Maggi and A.~Pratelli, \cite{MR2538501}. It gives fully explicit constants in the subcritical regime. The norms involved in the corrective term are not of the same nature.

\medskip Let us list a series of remarks which help for the understanding of our results.

\smallskip\noindent (i) \emph{Scaling properties of the Barenblatt profiles.} Consider the scaling $\lambda\mapsto u_\lambda$ with $u_\lambda(x):=\lambda^d\,u(\lambda\,x)$ for any $x\in\R^d$. Then we have
\[
\sigma_\lambda:=\frac 1{K_M}\ird{|x|^2\,u_\lambda}=\frac 1{\lambda^2}\,\frac 1{K_M}\ird{|x|^2\,u}=\frac\sigma{\lambda^2}
\]
and may observe that
\[
B_{\sigma_\lambda}(x)=\lambda^d\,B_\sigma(\lambda\,x)\;.
\]
As a consequence, we find that $\mathcal F_\sigma[u_\lambda]=\lambda^{d\,(m-1)}\,\mathcal F_\sigma[u]$. 

\smallskip\noindent (ii) \emph{Homogeneity properties of the Barenblatt profiles.} Similarly notice that for any $m\in(m_1,1)$, we have $C_M=C_1\,M^{-\frac{2\,(1-m)}{d\,(m-m_c)}}$ and $K_M=K_1\,M^{1-\frac{2\,(1-m)}{d\,(m-m_c)}}$. Let $u_\lambda:=\lambda\,u$ and denote by $B_{\sigma_\lambda}$ the corresponding best matching Barenblatt function. Using the fact that \hbox{$\nrm{u_\lambda}1=\lambda\,M$} if \hbox{$\nrm u1=M$} and observing that
\[
K_{\lambda M}=K_M\,\lambda^{1-\frac{2\,(1-m)}{d\,(m-m_c)}}\quad\mbox{and}\quad\ird{|x|^2\,u_\lambda}=\lambda\ird{|x|^2\,u}\;,
\]
we find
\[
\sigma_\lambda=\frac 1{K_{\lambda M}}\ird{|x|^2\,u_\lambda}=\lambda^{\frac{2\,(1-m)}{d\,(m-m_c)}}\,\sigma\;.
\]
Since $C_{\lambda M}=\lambda^{-\frac{2\,(1-m)}{d\,(m-m_c)}}\,C_M$, we find that
\[
B_{\sigma_\lambda}(x)=\(\lambda^{\frac{2\,(1-m)}{d\,(m-m_c)}}\,\sigma\)^{-\frac d2}\,\(\lambda^{-\frac{2\,(1-m)}{d\,(m-m_c)}}\,C_M+\tfrac{|x|^2}{\lambda^{\frac{2\,(1-m)}{d\,(m-m_c)}}\,\sigma}\)^\frac 1{m-1}=\lambda\,B_\sigma(x)\;.
\]
As a consequence, we find that $\mathcal F_\sigma[u_\lambda]=\lambda^m\,\mathcal F_\sigma[u]$.

\smallskip\noindent (iii) \emph{The $m=1$ limit.} As $m\to1$, which also corresponds to $p\to1$, we observe that the constant $\mathsf C_{p,d}$ in Theorem~\ref{Thm:GN} has a finite limit. Hence we get no improvement by dividing the improved Gagliardo-Nirenberg inequality by \hbox{$(p-1)$} and passing to the limit $p\to1_+$, since $\mathcal R^{(p)}[f]=O(p-1)$. By doing so, we simply recover the logarithmic Sobolev inequality as in \cite{MR1940370}.

This is consistent with the fact that, as $m\to 1_-$, we have $C_{m,d}\sim(1-m)^2$, $\sigma=O(K_M^{-1})=O(1-m)$ and, since
\[
B_\sigma(x)\sim B_0(x):=M\(\frac{d\,M}{2\,\pi\,\ird{|x|^2\,u}}\)^\frac d2\,\exp\(-\frac d2\,\frac M{\ird{|x|^2\,u}}\,|x|^2\)\,,
\]
we also get that $\mathcal F_\sigma[u]\sim\ird{u\,\log\(\frac u{B_0}\)}$. Hence, in Theorem~\ref{Thm:IN}, the additional term in \eqref{Eqn:improved} is of the order of $1-m$ and disappears when passing to the limit $m\to 1_-$.

%%%%%%%%%%%%%%%%%%%%%%%%%%%%%%%%%%%%%%%%%%%%%%%%%%%%%%%%%%%%%%%%%%%%%%
%%%%%%%%%%%%%%%%%%%%%%%%%%%%%%%%%%%%%%%%%%%%%%%%%%%%%%%%%%%%%%%%%%%%%%
\appendix\section{Computation of the constants}\label{Sec:Appendix}

Let us recall first some useful formulae. The surface of the $d-1$ dimensional unit sphere $\sphere$ is given by $|\sphere|=2\,\pi^{d/2}/\,\Gamma(d/2)$. Using the integral representation of Euler's Beta function (see \cite[6.2.1 p.~258]{MR0167642}), we have
\[
\ird{\(1+|x|^2\)^{-a}}=\pi^\frac d2\,\frac{\Gamma\big(a-\frac d2\big)}{\Gamma(a)}\;.
\]
With this formula in hand, various quantities associated with \emph{Barenblatt functions} can be computed. Applied to the function $B(x):=\(1+|x|^2\)^\frac 1{m-1}$, $x\in\R^d$, we find that
\[\label{Eqn:Mstar}
M_*:=\ird B=\pi^\frac d2\,\frac{\Gamma\big(\frac{d\,(m-m_c)}{2\,(1-m)}\big)}{\Gamma\big(\frac 1{1-m}\big)}\;.
\]
Notice that when $M=M_*$, $B=B_1$ with the notation~\eqref{Eqn:Barenblatt} of Section~\ref{Sec:CK}. As a consequence, for $B_1(x)=\(C_M+|x|^2\)^{\frac 1{m-1}}$, a simple change of variables shows that
\[
M:=\ird{B_1}=\ird{\(C_M+|x|^2\)^\frac 1{m-1}}=M_*\,C_M^{-\frac{d\,(m-m_c)}{2\,(1-m)}}\;,
\]
which determines the value of $C_M$, namely
\[
C_M=\(\frac{M_*}M\)^\frac{2\,(1-m)}{d\,(m-m_c)}\,.
\]
A useful equivalent formula is $C_M=C_1\,M^{-\frac{2\,(1-m)}{d\,(m-m_c)}}$ where $C_1=M_*^\frac{2\,(1-m)}{d\,(m-m_c)}$.

\medskip By recalling \eqref{IPP} and observing that
\[
\ird{B_1^m}=\ird{B_1^{m-1}\,B_1}=\ird{(C_M+|x|^2)\,B_1}=M\,C_M+K_M
\]
where $K_M:=\ird{|x|^2\,B_1}$, using $M\,C_M=C_1\,M^\gamma$ with $\gamma=\frac{(d+2)\,m-d}{d\,(m-m_c)}$, we find that
\be{Enq:KMCM}
K_M=\frac{d\,(1-m)}{(d+2)\,m-d}\,C_1\,M^\gamma\quad\mbox{and}\quad\ird{B_1^m}=\frac{2\,m}{(d+2)\,m-d}\,C_1\,M^\gamma\,.
\ee

\medskip Consider the sub-family of \emph{Gagliardo-Nirenberg-Sobolev inequalities}~\eqref{Ineq:GN}. It has been established in \cite[Theorem 1]{MR1940370} that optimal functions are all given by \eqref{Eqn:Optimal}, up to multiplications by a constant, translations and scalings. This allows to compute $\mathcal C_{p,d}^{\rm GN}$. All computations done, we find
\[
\mathcal C_{p,d}^{\rm GN}=\textstyle\(\tfrac{(p-1)^{p+1}}{(p+1)^{d+1-p(d-1)}}\)^\eta\,\(\tfrac{d+2-p\,(d-2)}{2\,(p-1)}\)^\frac 1{2\,p}\,\(\frac {\Gamma\(\frac{p+1}{p-1}\)}{(2\,\pi\,d)^\frac d2\,\Gamma\(\frac{p+1}{p-1}-\frac{d}{2}\)}\)^{(p-1)\,\eta}
\]
with $1/\eta=p\,(d+2-p\,(d-2))$.

It is easy to relate $\mathcal C_{p,d}^{\rm GN}$ and $\K{p,d}$. As in~\cite{MR1940370}, apply \eqref{Ineq:GN-NonHom0} to $f_\lambda$ such that $f_\lambda(x)=\lambda^\frac d{2\,p}\,f(\lambda\,x)$ for any $x\in\R^d$. With $a:=\ird{|\nabla f|^2}$, $b:=\ird{|f|^{p+1}}$, $\alpha:=\frac dp+2-d$ and $\beta:=d\,\frac{p-1}{2\,p}$, Inequality~\eqref{Ineq:GN-NonHom0} amounts to
\[
a\,\lambda^\alpha+b\,\lambda^{-\beta}\ge\K{p,d}\(\ird{|f|^{2\,p}}\)^\gamma\,.
\]
Optimizing the left hand side with respect to $\lambda>0$ shows that
\[
\K{p,d}\,\(\mathcal C_{p,d}^{\rm GN}\)^{2\,p\,\gamma}=\frac{\alpha+\beta}{\alpha^\frac\alpha{\alpha+\beta}+\beta^\frac\beta{\alpha+\beta}}\;.
\]

Let us consider \eqref{PreCpd}. With $p=\frac 1{2\,m-1}$, that is, $m=\frac{p+1}{2\,p}$, and $\mathcal F[u]=\frac m{1-m}\,\mathcal R^{(p)}[f]$ with $u=f^{2\,p}$, it is straightforward to check that
\[
\K{p,d}=\(\tfrac{2\,m-1}{1-m}\)^2\,\tfrac{d\,(m-m_c)}{(d+2)\,m-d}\,\frac{M_*^{1-\gamma}}{\sigma_*^{d\,(m-m_1)}}=\tfrac4{(p-1)^2}\,\tfrac{d-p\,(d-4)}{d+2-p\,(d-2)}\,M_*^{1-\gamma}\,\sigma_*^{d\,\frac{p-1}p-1}
\]
since $u=B_\sigma$ always provides the equality case. Hence, using Identity \eqref{Sigma*}, Inequality~\eqref{PreCpd} amounts to
\begin{multline*}
\ird{|\nabla f|^2}+\ird{|f|^{p+1}}-\K{p,d}\(\ird{|f|^{2\,p}}\)^\gamma\\
\ge\tfrac{(2\,m-1)^2}{m\,(1-m)}\,\sigma_*^{-\frac d2\,(m-m_c)}\,\tfrac{d^3\,(m-m_c)\,(m-m_1)\,(1-m)^2}{8\,m\,K_1}\,\frac{\(\frac m{1-m}\,\mathcal R^{(p)}[f]\)^2}{M^\gamma\,\sigma_*^{\frac d2\,(1-m)}}\;.
\end{multline*}
Using $K_1=\frac{d\,(1-m)}{(d+2)\,m-d}\,M_*^{1-\gamma}$ and expressing everything in terms of $p$, we finally get
\begin{multline*}
\mathsf C_{p,d}=\tfrac{(2\,m-1)^2}{8\,(1-m)^2}\,((d+2)\,m-d)\,d^2\,(m-m_c)\,(m-m_1)\,\frac{M_*^{\gamma-1}}{\sigma_*}\\
=\tfrac{(d-p\,(d-4))\,(d-p\,(d-2))\,(d+2-p\,(d-2))}{16\,p^3\,(p-1)^2}\,\frac{M_*^{\gamma-1}}{\sigma_*(p)}\;.
\end{multline*}

%%%%%%%%%%%%%%%%%%%%%%%%%%%%%%%%%%%%%%%%%%%%%%%%%%%%%%%%%%%%%%%%%%%%%%
%%%%%%%%%%%%%%%%%%%%%%%%%%%%%%%%%%%%%%%%%%%%%%%%%%%%%%%%%%%%%%%%%%%%%%
\medskip\begin{spacing}{0.8}\noindent{\it Acknowledgments.\/} {\small The authors acknowledge support both by the ANR projects CBDif-Fr and EVOL (JD) and by MIUR project ``Optimal mass transportation, geometrical and functional inequalities with applications'' (GT). The warm hospitality of the Department of Mathematics of the University of Pavia, where this work have been partially done, is kindly acknowledged. JD thanks Michael Loss for enlighting discussions on the entropy - entropy production method. Both authors thank D.~Matthes for his careful reading of an earlier version of the manuscript, which has resulted in a significant improvement of our results.}
\par\medskip\noindent{\scriptsize\copyright\,2012 by the authors. This paper may be reproduced, in its entirety, for non-commercial purposes.}\end{spacing}

%%%%%%%%%%%%%%%%%%%%%%%%%%%%%%%%%%%%%%%%%%%%%%%%%%%%%%%%%%%%%%%%%%%%%%
%%%%%%%%%%%%%%%%%%%%%%%%%%%%%%%%%%%%%%%%%%%%%%%%%%%%%%%%%%%%%%%%%%%%%%
%\nocite*
%\bibliographystyle{siam}\small
%\bibliography{DT}

\def\cprime{$'$}

\end{document}